\documentclass[a4,9pt]{amsart}
\usepackage{amssymb,amstext,amsmath,amscd,amsthm,amsfonts,enumerate,graphicx,latexsym, hyperref,cleveref,mathtools,mathrsfs,stmaryrd}
\usepackage[usenames]{color}
\usepackage[all]{xy}

 \usepackage{comment}
 \usepackage{mathrsfs}
\usepackage{tikz-cd}
\usepackage{tikz}
\usepackage{xcolor}                                                                                                                                                                                    
\usepackage{changepage}
\usetikzlibrary{shapes,arrows,fit,calc,positioning}
\usetikzlibrary{matrix}
\xyoption{all}
\usepackage{mathtools}
\newtheorem{thm}{Theorem}[section]

\newtheorem{cor}[thm]{Corollary}

\newtheorem{lem}[thm]{Lemma}
\newtheorem{prop}[thm]{Proposition}

\newtheorem{theorem}[thm]{Theorem}
\theoremstyle{definition}
\newtheorem{dfn}[thm]{Definition}
\newtheorem{rem}[thm]{Remark}

\newtheorem{ex}[thm]{Example}

\newtheorem{obs}[thm]{Observations}

\newtheorem*{claim*}{Claim}
\theoremstyle{remark}

\numberwithin{equation}{thm}

\def\ker{\operatorname{Ker}}

\def\m{\mathfrak{m}}

\newcommand{\msa}{\mathscr{A}_1} 
\newcommand{\swp}{S^{p\wedge p^2}} 
\newcommand{\nn}{\textrm{NNL}_1} 
\renewcommand{\psi}{\tau}
\renewcommand{\kappa}{\omega}

\newcommand{\gj}{\tilde{f_j}} 
\newcommand{\go}{\tilde{f_1}}
\newcommand{\gt}{\tilde{f_t}}
\newcommand{\gr}{\tilde{f_r}} 
\newcommand{\projdim}{\operatorname{proj\,dim}}

\newcommand{\whp}{\frac{\omega^{p-1}+\cdots + h^{p-1}}{p}}
\newcommand{\wh}{\omega^{p-1}+\cdots + h^{p-1}}
\newcommand{\WH}{W^{p-1}+\cdots + h^{p-1}}

\newcommand{\om}{\omega} 

\newcommand{\un}{\underline}

\newcommand{\sqn}{\sqrt[n]{f}} 

\begin{document}
\setlength{\baselineskip}{15pt}
\title{On the integral closure of radical towers in mixed characteristic}
\author{Daniel Katz}
\address{Department of Mathematics, University of Kansas, Lawrence, KS 66045-7523, USA}
\email{dlk53@ku.edu}

\author{Prashanth Sridhar}
\address{Charles University, Faculty of Mathematics and Physics, Ke Karlovu 3, 121 16 Prague 2, Czech Republic}
\email{sridhar@karlin.mff.cuni.cz}
\begin{abstract}{We study the Cohen-Macaulay property of a particular class of radical extensions of an unramified regular local ring having mixed characteristic.}

\end{abstract}
\subjclass[2010]{}
\maketitle

\section{Introduction} In this note we consider the integral closure of certain radical towers in mixed characteristic $p>0$. The motivation for our work is the following. Suppose $S$ is an integrally closed Noetherian domain with quotient field $L$ and $f\in S$ is square-free, i.e., $fS_Q = QS_Q$, for all height one primes $Q\subseteq S$ containing $f$. It is well-known that if the natural number $n\in S$ is a unit, then $R := S[\sqrt[n]{f}]$ \footnote{Throughout this paper, when we write $\sqrt[n]{f}$ for an element $f\in S$, we simply mean that $f$ is a root of $X^n -f$, where $X$ is an indeterminate over $S$.}  is integrally closed\footnote{To see that $S[\sqn]$ is integrally closed, set $g(X) = X^n-f$ and note that since $S$ has characteristic zero, the extension of quotient fields derived from $S\subseteq S[\sqn]$ is separable, and thus $g'(\sqn)$, and hence $nf$, multiplies the integral closure of $S[\sqn]$ into $S[\sqn]$. Since $n$ is a unit, one just has to check that $S[\sqn]_P$ is a DVR for any height one prime $P\subseteq S[\sqn]$ containing $f$ (since $S[\sqn]$ is free over $S$, and therefore satisfies Serre's condition $S_2$). But for such a prime, it is easy to see that $P_P = (\sqn)_P$ since $f$ is square-free in $S$ and thus $P_{P\cap S} = (f, \sqn)S_{P\cap S} [\sqn]$.}. If $f_1, \ldots, f_r \in S$ are square-free and no two elements in the sequence are contained in the same height one prime of $S$, then $R := S[\sqrt[n]{f_1}, \ldots, \sqrt[n]{f_r}]$ is integrally closed (see \cite{HK}). In each of these cases $R$ is a free $S$-module, so that if $S$ is a Cohen-Macaulay ring, then $R$ is also Cohen-Macaulay. In other words, maintaining the same assumptions, the integral closure of $S$ in the given radical extension of $L$ is Cohen-Macaulay. When $n$ is not a unit in $S$, the ring $S[\sqrt[n]{f_1}, \ldots, \sqrt[n]{f_r}]$ need not be integrally closed. However, if $S$ is regular, and $R$ denotes the integral closure of $S$ in $L(\sqrt[n]{f_1}, \ldots, \sqrt[n]{f_r})$, it may be that $R$ is Cohen-Macaulay, though it can be difficult to determine when $R$ is Cohen-Macaulay. For example, if $S$ is an unramified regular local ring of mixed characteristic $p > 0$ and  $f\in S$ is square-free, then for $R$ the integral closure of $S[\sqrt[p]{f}]$, $R$ is Cohen-Macaulay (see \cite{DK1}, Lemma 3.2). Unfortunately, this fails even when adjoining $p$th roots of two square-free elements, see for example \cite{PS1}, Example 2.12 or \cite{PS2}, Example 4.8. The difficulties in adjoining more than one $p$th root stem from the behavior of the $f_i$ modulo $pS$. However, as observed in \cite{DK1}, \cite{PS1} and \cite{PS2} there  is good behavior when the elements are $p$-th powers modulo $p^2S$. This is largely due to the existence of a unique unramified prime $P\subseteq R$ lying over $pS$ such that the corresponding extension of residue fields is trivial. The purpose of this note is to generalize the results from \cite{DK1}, \cite{PS1}, and \cite{PS2} beyond the biradical case to what we call {\it class one radical towers}, i.e., radical towers where the image of each $f_i$ is a $p$th power in $S/p^2S$. Note that to obtain good results for general $p$-th root towers, it already suffices to understand the case when the elements are $p$-th powers modulo $pS$ (see \Cref{pwedge_p2}). Here is the main result of this paper. Note that we allow $S$ to be more general than an unramified regular local ring of mixed characteristic.

\begin{thm} Let $S$ be an integrally closed Noetherian domain and $p>0$ a prime integer that is a non-zero prime in $S$. Suppose $f_1, \ldots , f_r \in S$ are square-free and there exist $h_i \in S$ such that $f_i \equiv h_i^p$ $\textrm{mod} \ p^2S$. Suppose further that no two $f_i$ are contained in the same height one prime of $S$. Let $n_1, \ldots, n_r > 1$ be integers such that $n_i = pd_i$ with each $d_i$ a unit in $S$, and write $R$ for the integral closure of $S[\sqrt[n_1]{f_1}, \ldots, \sqrt[n_r]{f_r}\ ]$. Then $R$ is a free $S$-module. In particular, if $S$ is Cohen-Macaulay or an unramified regular local ring, $R$ is Cohen-Macaulay.
\end{thm}

The main focus of the proof of this theorem is when each $n_i = p$, so that we get a proper generalization of the results in \cite{DK1}, \cite{PS1}, \cite{PS2}. We are able to reduce to this case by invoking the results that hold when the degree of the radical extension is a unit. In what follows, section two deals with our conventions and preliminary results, while our main results are presented in section three. In section four, we present some applications of the main result in section three. In particular, in Theorem 4.1, we show that if $R$ is the integral closure of an unramified regular local ring $S$ of mixed characteristic in a particular type of radical tower that is not a class one tower, then $R$ admits a small Cohen-Macaulay algebra, even if $R$ itself is not Cohen-Macaulay. 
\par 
\textbf{Acknowledgements.} Prashanth Sridhar was supported by the grant GA ČR 20-13778S from the Czech Science Foundation.

\section{Preliminaries} All rings considered in this paper are commutative and Noetherian. Throughout this paper $S$ will denote an integrally closed Noetherian domain in which the prime $p \in \mathbb{Z}$ is a non-zero prime element in $S$. We will use $L$ to denote the quotient field of $S$. Though our primary interest is in the case that $S$ is an unramified regular local ring having mixed characteristic $p > 0$, our main results hold under more general conditions, so that the cases when $S$ is an unramified regular local ring often appear as secondary statements. For the remainder of the paper, we will assume that the non-zero, non-unit elements $f_1, \ldots, f_r\in S$ are square-free. We will also use the following notation throughout the rest of this paper. 

\medskip
\noindent
{\bf Definitions 2.1.} For $S$ and $f_1, \ldots, f_r\in S$ as above:
\begin{enumerate}
\item[(i)] The elements $f_1, \ldots, f_r \in S$ are said to satisfy $\msa$ if no two of these elements are contained in the same height one prime of $S$. 
\item[(ii)] We will use $\swp$ to denote the set of elements $f\in S$ that are $p$-th powers modulo $p^2S$. Note that this is a multiplicatively closed subset of $S$, and that if $f\in S^{p\wedge p^2}$ is square-free, then $p\nmid f$. We will write $S^{p\wedge p}$ for the subring of $S$ that are $p$th powers modulo $pS$. 
\item[(iii)] For square-free elements $f_1, \ldots, f_r \in S$, we call $L(\sqrt[n_1]{f_1}, \ldots, \sqrt[n_r]{f_r})$ a {\it a class one radical tower} if: \begin{enumerate}
\item[(a)] The set $\{f_1, \ldots, f_r\}$ satisfies $\msa$ and 
\item[(b)] Each $f_i \in \swp$. 
\end{enumerate}
\item[(iv)] For a commutative ring $R$ of dimension at least one, we use the notation \[\nn(R):=\{P\in Spec(R)\: | \: height(P)=1,\: R_P\text{ is not a DVR}\}\]
\item[(v)]  For a Noetherian ring $R$, a prime ideal $Q\subseteq R$ and $q_1, \ldots, q_r \in Q$, we will use $(Q\:|\:q_1,\dots,q_r)$ to indicate that $Q_Q=(q_1,\dots,q_r)_Q$.
\item[(vi)] Throughout this paper, in any discussion involving $f_1, \ldots, f_r\in S$, and $n_1, \ldots, n_r\geq 2$, we set $A := S[\sqrt[n_1]{f_1}, \ldots, \sqrt[n_r]{f_r}]$, $K := L(\sqrt[n_1]{f_1}, \ldots, \sqrt[n_r]{f_r})$ and let $R$ denote the integral closure of $A$, or equivalently, the integral closure of $S$ in $K$. 
\end{enumerate} 

\medskip
For ease of reference, we include the following proposition which can essentially be found as separate Propositions 5.2 and 5.3 in \cite{HK}. In these propositions from \cite{HK}, it is assumed all roots have the same order. However, it is not difficult to see that the same proofs give rise to the following statement. 
\begin{prop}\label{HK}
Let $S$ be an integrally closed Noetherian domain and $n = n_1\cdots n_r\in S$ a unit for some positive integer $n$. Let $f_1,\dots,f_r\in S$ be square-free elements satisfying $\mathscr{A}_1$. Then 
\begin{enumerate}
\item[(i)] $f_2,\dots,f_r$ are square-free and satisfy $\msa$ in $S[\sqrt[n_1]{f_1}]$.
\item[(ii)] $R=S[\sqrt[n_1]{f_1},\dots,\sqrt[n_r]{f_r}]$ is integrally closed.
\end{enumerate} 
\end{prop}

We will make considerable use of the following observations. 
\begin{obs}\label{correctdegree} (i)  Let $f_1, \ldots, f_r\in S$ be square free elements. Set $A := S[\sqrt[n_1]{f_1}, \ldots,\sqrt[n_r]{f_r}]$ and assume $\{n, f_1, \ldots, f_r\}$ satisfies $\msa$, where $n = n_1\cdots n_r$.  Then $[K:L] = n$, where $K$ denotes the quotient field of $A$. Moreover, $A[\frac{1}{n}]$ is integrally closed. 

\medskip
\noindent
(ii)   Let $f_1, \ldots, f_r\in S$ be square free elements. Set $A := S[\sqrt[n_1]{f_1}, \ldots,\sqrt[n_r]{f_r}]$ and assume $\{n, f_1, \ldots, f_r\}$ satisfies $\msa$, where $n = n_1\cdots n_r$.  Let $Q$ denote the kernel of the natural homomorphism  from $S[X_1, \ldots, X_r]$ to $A$, where the $X_i$ are indeterminates over $S$. Then $Q = \langle X_1^{n_1}-f_1, \ldots, X_r^{n_r}-f_r\rangle$. 

\medskip
\noindent
(iii) Let $S\subseteq C\subseteq D$ be extension of rings such that $D$ is integral over $S$ and is a domain. If $C$ is regular in codimension one and $D$ is birational to $C$, then $D$ is regular in codimension one.
\end{obs}

\begin{proof} For (i), note that if $n\geq 1$ and $T$ is an integrally closed domain of characteristic zero and $f \in T$ is square-free in $T$, then $f$ is not a $q$th power in the quotient field of $T$, for any prime $q$ dividing $n$. Thus, by Theorem 9.1 in \cite{L}, $X^n-f$ is irreducible over the quotient field of $T$ (equivalently over $T$), where $X$ is an indeterminate over $T$.  Otherwise, $X^n-f$ has a root in $T$, in which case, $f$ is not square-free in $T$. Now let $K_i$ denote the quotient field of $S_i := S[\sqrt[n_1]{f_1}, \ldots, \sqrt[n_i]{f_i}]$. By what we have just shown, $[K_1:L] = n_1$. Proceeding by induction on $i$, it suffices to show that $X_i^{n_i}-f_i$ is the minimal polynomial of $\sqrt[n_i]{f_i}$ over $K_{i-1}$, where $X_i$ is an indeterminate over $K_{i-1}$. For this, by the $\msa$ assumption, there is no harm in inverting $n$. But then, by Proposition \ref{HK}, $S_{i-1}[\frac{1}{n}]$ is integrally closed and $f_i$ is square-free in $S_{i-1}[\frac{1}{n}]$, and thus, $X_i^{n_i}-f_i$ is irreducible over $K_i$, which is what we want. It follows now that $[K:L] = n$. The second statement follows immediately from Proposition \ref{HK}. 

\medskip
\noindent
For (ii), we will use the following fact, which follows easily from the division algorithm. Suppose $T$ is an integral domain with quotient field $F$, $E$ a field extension of $F$ and $\alpha \in E$ algebraic over $F$. Let $p(x)$ be the (monic) minimal polynomial of $\alpha$ over $F$ and assume $p(x) \in T[x]$. Then 
$T[\alpha] \cong T[x]/\langle  p(x)\rangle$. Now, proceeding by induction on $r$, if $r = 1$, we have that $S[\sqrt[n_1]{f_1}] \cong S[X_1]/\langle X_1^{n_1}-f_1\rangle$, since, as in part (i), $X_1^{n_1}-f_1$ is the minimal polynomial for $\sqrt[n_1]{f_1}$ over $L$. From the proof of part (i), we also have that $X_r^{n_r}-f_r$ is the minimal polynomial of $\sqrt[n_r]{f_r}$ over the quotient field of $B := S[\sqrt[n_1]{f_1}, \ldots, \sqrt[n_{r-1}]{f_{r-1}}]$. Thus, $B[\sqrt[n_r]{f_r}]$ is isomorphic to $B[X_r]/\langle X_r^{n_r}-f_r\rangle$. Since $B$ is isomorphic to $S[X_1, \ldots, X_{r-1}]/\langle X_1^{n_1}-f_1, \ldots, X_{r-1}^{n_{r-1}}-f_{r-1}\rangle$, this gives us what we want.

\medskip
\noindent
For (iii), let $P$ be a height one prime in $D$, $P_0 := C\cap P$ and $Q := S\cap P$. Since $S\subseteq D$ satisfies going down, $Q$ has height one. Since $C$ is integral over $S$, $P_0$ has height one. Thus, $C_{P_0}$ is a DVR. Since $C$ and $D$ are birational, and $C_{P_0} \subseteq D_P$, we have $C_{P_0} = D_P$, so $D$ is regular in codimension one.
\end{proof} 

We include the following from \cite{DK1} and \cite{PS2} for easy reference:
\begin{lem}[\cite{DK1}]\label{P6}
Let $p\geq 3$ and write $p=2k+1$. For $h\in S\setminus pS$ and $W$ an indeterminate over $S$, if
\begin{equation}
    C :=(W-h)^p-(W^p-h^p)=\sum_{j=1}^k (-1)^{j+1}{p\choose j}(W\cdot h)^j[W^{p-2j}-h^{p-2j}]
\end{equation}
$C' := C\cdot (p(W-h))^{-1}$ and $\tilde{P}:=(p,W-h)S[W]$, then $C'\notin \tilde{P}$. 
\end{lem}

\begin{lem}[\cite{PS2}]\label{l_imodP}
Let $p\geq 3$ and write $p=2k+1$. For $h\in S\setminus pS$ and $W$ an indeterminate over $S$, suppose $C'$ is as defined in \ref{P6}. Then $C' \equiv h^{p-1}\text{ mod }(p,W-h)S[W]$.
\end{lem}

\medskip

Our work in section three relies heavily on Lemma 3.2 in \cite{DK1}. Let $f\in S$ be square-free such that $f\in \swp$, say $f = h^p+p^2g$. Take $\omega$ satisfying $\omega ^p = f$, and set $\tau := \whp = \frac{pg}{\omega-h}$. Then Lemma 3.2 from \cite{DK1} shows that $R$, the integral closure of $S[\omega]$, equals $S[\omega, \tau]$ and that $R$ is a free $S$-module. The proof relies on showing that $Q_1 := (p,\omega-h, \tau)$ and $Q_2 := (p, \omega-h, \tau-c')$ are the only height two primes in $S[\omega, \tau]$ containing $p$, where $c'$ is the image in $S[\omega]$ of the element $C'$ in \Cref{P6} under the map sending $W\rightarrow \omega$. This is done by noting that $\tau$ satisfies $l(T) := T^2-c'T -g(\omega-h)^{p-2}$ when $p > 2$ and that if $\widetilde{Q}$ is a height two prime in $S[\omega, T]$ containing $l(T)$, then $\widetilde{Q_1}$ equals $\widetilde{Q_1} := (p, \omega-h, T)S[\omega, T]$ or $\widetilde{Q_2} := (p, \omega-h, T-c')S[\omega, T]$. However, the proof neglected to show that the kernel of the natural map from $S[\omega, T]$ to $S[\omega, \tau]$ is contained in these primes. The next lemma closes this small gap.

\begin{lem}\label{taukernel} Let $f\in \swp$ be square-free, $f = h^p+p^2g$ and $\omega ^p = f$ a $p$-th root in a field extension of $L$. Set $\tau = \frac{\omega^{p-1}+\cdots +h^{p-1}}{p} = \frac{pg}{\omega-h}$. Consider the natural map $\phi : S[\omega, T] \to S[\omega, \tau]$, where $T$ is an indeterminate over $S[\omega]$. Then $\textrm{Ker}(\phi)$ is generated by $l(T) := T^2-c'T -g(\omega-h)^{p-2}, m(T) := pT-(\wh), n(T) := (\omega-h)T-pg$, if $p > 2$ and $l_0(T) = T^2-hT-g, m(T), n(T)$, if $p = 2$. In particular, for $\widetilde{Q_1}$ and $\widetilde{Q_2}$ in the paragraph above, $\widetilde{Q_1}$ and $\widetilde{Q_2}$ contain $\ker(\phi)$. 
\end{lem}

\begin{proof} Suppose $p > 2$. The proof of Lemma 3.2 in \cite{DK1} shows that $l(T) \in \ker(\phi)$. Thus, we need only find linear generators in $\ker(\phi)$. Suppose $a(\om)T-b(\om) \in \ker(\phi)$. By definition, \[a(\om)(\wh)-b(\om)p = 0\] in $S[\om]$. Noting that $S[\om] = S[W]/\langle W^p-f\rangle$, it follows that 
\[
a(W) (\WH) -b(W)p = q(W)(W^p-f),
\]
for some $q(W) \in S[W]$. Therefore, 
\[
a(W)(\WH) -b(W)p -q(W)(W^p-h^p-p^2g) = 0
\]
and thus
\[\{a(W) -q(W)(W-h)\}(\WH) + \{-b(W)+pgq(W)\}p = 0.
\] Since $\WH, p$ form a regular sequence in $S[W]$, we have 
\[
a(W) = \lambda (W) (-p) + q(W)(W-h) \quad\textrm{and}\quad -b(W) = \lambda (W)(\WH) + q(W) (-pg).
\]  for some $\lambda (W)\in S[W]$. It follows that in $S[\om, T]$,
\[
a(\om)T-b(\om) = -\lambda(\omega)\cdot m(T) + q(\om)\cdot n(T), 
\] which is what we want. If $p = 2$, the proof is similar (though, in fact, easier), as \cite{DK1}, Lemma 3.2 shows that $l_0(T)$ belongs to $\ker(\phi)$. 
For the second statement, $l(T)$ and $n(T)$ are clearly contained in each $\widetilde{Q_i}$. That $m(T)$ is contained in each $\widetilde{Q_i}$ follows from the fact that we can write $\omega^{p-1}+\cdots + h^{p-1}$ as $(\omega-h)\cdot g(\omega) +ph^{p-1}$, for some $g(\omega)\in S$. 
\end{proof}

\section{Class one radical towers}\label{c1rt} In this section we present our main result. The crucial case for this result is Theorem \ref{class1thm} below, which deals with the case of adjoining $p$th roots of square-free elements. Before addressing Theorem \ref{class1thm}, we need a few preliminary results. 

\begin{lem}\label{kernel_lem}
Let $f_1, \ldots, f_r \in S$ be such that $K := L(\sqrt[p]{f_1}, \ldots, \sqrt[p]{f_r})$ is a class one radical extension of $L$.  Set $\omega_i: = \sqrt[p]{f_i}$. Let $R_i$ denote the integral closure of $S[\om _i]$ and write $V$ for the join of $R_1, \ldots, R_r$. Recalling from the previous section that $R_i = S[\omega_i, \tau_i]$, for $\tau _i$ as in  Lemma \ref{taukernel}, let $W_i,T_i$ be indeterminates over $S$ and let $\phi_i: S[W_i, T_i]\rightarrow R_i$ be the natural surjection. Write $A_i$ for the kernel of $\phi_i$. Then $V$ is a free $S$-module and the kernel of the natural homomorphism from $S[W_1, T_1, \ldots, W_r, T_r]$ to $V$ is generated by $A_1+\cdots + A_r$.
\end{lem} 

\begin{proof} We let $K_i$ denote the quotient field of $S[\omega_i]$, so that $K = K_1\cdots K_r$ and $K_1\cdots K_r$ is the compositum of the $K_i$. We induct on $r$. Suppose $r = 2$. Since $[K:L] = p^2$ by \Cref{correctdegree}(i), $K_1$ and $K_2$ are linearly disjoint over $L$ and a basis for $K$ over $L$ is obtained by taking the products of the basis elements for each $K_i$ over $L$. Since each $R_i$ is free over $S$, we may take bases from each $R_i$ to serve as the bases for $K_i$. Thus, the product of the bases for the $R_i$ are linearly independent over $S$. Since these products span $V$, we have that $V$ is a free $S$-module. Thus, the canonical map $R_1\otimes _S R_2\to V$ is a surjection of finite free $S$-modules of the same rank and hence an isomorphism. Moreover, this map is a ring homomorphism, so that it is an isomorphism of $S$-algebras. Thus the natural map from $(S[W_1, T_1]/A_1)\otimes_S (S[W_2,T_2]/A_2) \to V$ is an isomorphism. Since $S[W_1, T_1]/A_1\otimes_S S[W_2,T_2]/A_2 \cong S[W_1, T_1, W_2, T_2]/\langle A_1+A_2\rangle$, this gives us what we want. 

\medskip
\noindent
The proof of the inductive step is essentially the same as the case $r = 2$. If $r > 2$, set $\tilde{K} = K_1\cdots K_{r-1}$ and let $\tilde{V}$ denote the join of $R_1, \ldots, R_{r-1}$, so that $[\tilde{K} :L] = p^{r-1}$, by Observation \ref{correctdegree}(i). By induction, $\tilde{V}$ is a free $S$-module, and the kernel of the natural map from $S[W_1, T_2, \ldots, W_{r-1}, T_{r-1}]$ to $\tilde{V}$ is generated by $A_1+\cdots +A_{r-1}$. Since $K = \tilde{K}\cdot K_r$ and $[K:L] = [\tilde{K}:L]\cdot [K_r:L]$, $\tilde{K}$ and $K_r$ are linearly disjoint over $L$. Now we can repeat the argument in the paragraph above on $\tilde{V}$ and $R_r$ to complete the proof, since $V$ is the join of $\tilde{V}$ and $R_r$. 
\end{proof}

\begin{lem}\label{ht1_primes_p}
Let $S$ be an integrally closed domain, and $p\geq 3$ a prime number that remains prime in $S$. Let $f_1, \ldots, f_r$ be square-free elements in $S$ and $\omega_i^p = f_i$ be $p$-th roots in some field extension of $L$.  Suppose that $K :=  L(\omega_1, \ldots, \omega_r)$ is a class one radical tower. Choose $h_i\in S$ such that $f_i\equiv h_i^p$ $\textrm{mod} \ p^2S$. Let $R_i$ denote the integral closure of $S[\omega_i]$ for $1\leq i\leq r$ and let $V$ be the join of the $R_i$. The following hold:
\begin{enumerate}
    \item[(i)] There are $2^r$ height one primes in $V$ containing $p$ and at most $2^r-r-1$ of them are singular.
    \item[(ii)] If $Q\subseteq V$ is a height one prime containing $p$, then the non-singular primes are either of the form $(Q\:|\:p)$ or $(Q\:|\:\kappa_i-h_i)$. The (possibly) singular ones are of the form $Q_{(i_1,\dots,i_l)}:=(Q\:|\:\kappa_{i_1}-h_{i_1},\dots,\kappa_{i_l}-h_{i_l})$ for some $\{i_1,\dots,i_l\}\subseteq \{1,\dots,r\}$, $i_1<\dots<i_l$ and $l\geq 2$.
\end{enumerate}
\end{lem}

\begin{proof} We retain the notation from the previous lemma and its proof. First note that there are $2^r$  height one primes in $V$ containing $p$. From \cite{DK1}, Lemma 3.2, we have that $Q_{i,1}:= (p, \omega_i-h_i, \tau_i)R_i$ and $Q_{i,2} :=(p, \omega_i-h_i, \tau_i-c'_i)R_i$ are the only height one primes in $R_i$ lying over $pS$. Moreover, from the proof of Lemma 3.2 in \cite{DK1}, for each $1\leq i\leq r$,we have:
\[
(*)\quad \quad Q_{i,1} = (Q_{i,1} \ | \omega_i-h_i) \quad \textrm{and}\quad Q_{i,2} = (Q_{i,2}\ |\ p).
\]
Let $\widetilde{Q_{i,j}}$ denote the pre-image of $Q_{i,j}$ in $S[W_i, T_i]$ under the natural map. From \Cref{taukernel}, $\widetilde{Q_{i,j}}=(p, W_i-h_i, T_i-q)$ for some $q\in S[W_i]$. It then follows from \Cref{ht1_primes_p} that the ideal generated by $\sum_{i=1}^n\widetilde{Q_{i,j_i}}$ contains the kernel of the natural map from $S[W_1, T_1, \ldots, W_r, T_r]$ to $V$ for any choice of $j_i\in \{1,2\}$. Any such ideal is clearly prime of height $2n+1$ and hence their images in $V$ account for $2^r$ distinct height one primes containing $p$. On the other hand, if $Q_0$ is a height one prime in $V$ containing $p$, then, for each $1\leq i\leq r$, $Q_0\cap R_i = Q_{i,j_i}$ for some $j_i \in \{1,2\}$, since $Q_0\cap R_i$ must be a height one prime containing $p$ ($R_i$ is integrally closed). Thus $Q_0$ contains the ideal generated by $\sum_{i=1}^n Q_{i,j_i}$ and thus $Q_0$ is one of the $2^r$ height one primes we have accounted for. Therefore, there are exactly $2^r$ height one primes in $V$ containing $p$.

\par Now, let $Q\subset V$ be a height one prime containing $p$. If $Q = \langle Q_{1,2}+\cdots + Q_{r,2}\rangle$, then it follows from (*) that $Q_Q = pV_Q$. 
If there is a single $i$ for which $\psi_i-c_i'\notin Q$, then it follows from (*)  that $\kappa_i-h_i$ is a local generator for $Q$. Thus, there are at most $2^r-r-1$ singularities in codimension one in $V$. Finally, if $\{i_1,\dots,i_l\}\subseteq \{1,\dots,r\}$ is such that $i_1<\dots<i_l$, $l\geq 2$ and $\psi_{i_j}-c'_{i_j}\notin Q$ for all $1\leq j\leq l$, then it follows that $Q$ is of the form $(Q\:|\:\kappa_{i_1}-h_{i_1},\dots,\kappa_{i_l}-h_{i_l})$. 
\end{proof}

\begin{rem}
Although the standing assumption in \cite{DK1} requires $S/pS$ to be integrally closed, note that the conclusion and proof of [\cite{DK1}, Lemma 3.2] holds without this additional hypothesis. In particular, if $f\in S^{p\wedge p^2}$ is square free, the conclusion and proof of loc.cit. holds without the hypothesis that $S/pS$ be integrally closed. 
\end{rem}
The following theorem is the crucial case for the main result of this paper. It extends the corresponding results from \cite{DK1}, \cite{PS1}, and \cite{PS2}. 

\begin{theorem}\label{class1thm} Let $S$ be an integrally closed domain and $p$ a prime integer that remains a non-zero prime in $S$. Let $f_1, \ldots, f_r$ be square-free elements in $S$ and $\omega_i^p = f_i$ be $p$-th roots in some field extension of $L$. Suppose that $K :=  L(\omega_1, \ldots, \omega_r)$ is a class one radical tower. Then $R$, the integral closure of $S$ in $K$, is a free $S$-module. In particular, if $S$ is Cohen-Macaulay or an unramified regular local ring, then $R$ is Cohen-Macaulay.
\end{theorem}

\begin{proof} The second statement follows immediately from the first, so we only concern ourselves with the general case. We first assume $p\geq 3$. Note that $R$ is the integral closure of the ring $S[\kappa_1, \ldots, \kappa_r]$. 
Let $R_i$ be the integral closure of $S[\kappa_i]$ and let $V\subseteq R$ denote the join of the $R_i$. Note that $V$ is $S$-free of rank $p^r$ by \Cref{kernel_lem}. For each $1\leq i\leq r$, choose $h_i$ so that $f_i-h_i^p\in p^2S$. From \Cref{ht1_primes_p}, we have the following data: (a) there are exactly $2^r$ height one primes in $V$ containing $p$ (b) If $Q\subseteq V$ is a height one prime containing $p$, then the non singular primes are either of the form $(Q\:|\:p)$ or $(Q\:|\:\kappa_i-h_i)$ (c) The (possibly) singular ones are of the form $Q_{(i_1,\dots,i_l)}:=(Q\:|\:\kappa_{i_1}-h_{i_1},\dots,\kappa_{i_l}-h_{i_l})$ for some $\{i_1,\dots,i_l\}\subseteq \{1,\dots,r\}$, $i_1<\dots<i_l$ and $l\geq 2$. (d) There are at most $2^r-r-1$ singularities in codimension one in $V$.
\par Note that from Observation \ref{correctdegree}(iii), $V$ is regular in codimension one outside of the $Q_{(i_1,\dots,i_l)}$. We now identify an $\mathrm{R}_1$-ification of $V$. For $i,j\in \{1,\dots,n\}$, $i<j$, define \[\eta_{ij}:=p^{-1}(\kappa_i-h_i)^{p-2}(\kappa_j-h_j)\in K.\]

\noindent
Then for $X$ an indeterminate over $V$, $\eta_{ij}$ satisfies the integral equation
\[v_{ij}(X):=X^{p-1}-(\psi_i-c'_i)^{p-2}(\psi_j-c'_j)\in V[X].\]

\noindent
To de-singularize $Q_{(i_1,\dots,i_l)}$ consider the finite birational extension $V\hookrightarrow V_{(i_1,\dots,i_l)}:=V[\eta_{i_1i_2},\dots,\eta_{i_1i_l}]$. From \ref{l_imodP} we have \footnote{If $R$ is a ring of characteristic $p$ and $X,Y$ indeterminates over $R$, $X^{p-1}-Y^{p-1}=\prod_{i=1}^{p-1} (X+iY)$.} 
\begin{align*}
 v_{i_1i_2}(X)  &\equiv X^{p-1}-(h_{i_1}^{p-2}h_{i_2})^{p-1} \text{ mod($Q_{(i_1,\dots,i_l)}V[X])$}
 \\
 &\equiv \prod_{k=1}^{p-1} (X+kh_{i_1}^{p-2}h_{i_2}) \text{ mod($Q_{(i_1,\dots,i_l)}V[X])$}
 \end{align*}
Since $S_{(p)}$ is universally catenary, it follows that height one primes in $V[\eta_{i_1i_2}]$ lying over $Q_{(i_1,\dots,i_l)}$ are of the form $$Q_{[(i_1,\hat{i_2},\dots,i_l)|\:k]}:=(Q_{(i_1,\dots,i_l)},\eta_{i_1i_2}+kh_{i_1}^{p-2}h_{i_2})$$ for $1\leq k\leq p-1$. 
The point is that $Q_{[(i_1,\hat{i_2},\dots,i_l)|\:k]}$ locally has one less generator: it is of the form $$(Q_{[(i_1,\tilde{i_2},\dots,i_l)|\:k]}\:|\:(\kappa_{i_1}-h_{i_1},\widehat{\kappa_{i_2}-h_{i_2}},\dots,\kappa_{i_l}-h_{i_l})).$$ 

\noindent
To see this, note that $\prod_{i=1, i\neq k}^{p-1}(\eta_{i_1i_2}+ih_{i_1}^{p-2}h_{i_2})\notin Q_{[(i_1,\hat{i_2},\dots,i_l)|\:k]}$ so that $\eta_{i_1i_2}+kh_{i_1}^{p-2}h_{i_2}$ is locally a redundant generator. Also, since $(\kappa_{i_1}-h_{i_1})\cdot \eta_{i_1i_2}=(\kappa_{i_2}-h_{i_2})(\psi_{i_1}-c'_{i_1})$ and $(\psi_{i_1}-c'_{i_1})\notin Q_{[(i_1,\hat{i_2},\dots,i_l)|\:k]}$, $\kappa_{i_2}-h_{i_2}$ is a redundant generator locally.
\\
Proceeding inductively, it is clear that in $V_{(i_1,\dots,i_l)}$ all height one primes lying over $Q_{(i_1,\dots,i_l)}$ (at most $(p-1)^{l-1})$ are non-singular. Now set \[\mathscr{R}V:=V[\{p^{-1}(\kappa_{i_1}-h_{i_1})^{p-2}(\kappa_{i_2}-h_{i_2})\}_{\{i_1,i_2\in\{1,\dots,r\}\:|\:i_1<i_2\}}]\]
For all possible $\{i_1,\dots,i_l\}\subseteq \{1,\dots,r\}$, $i_1<\dots<i_l$ and $l\geq 2$, $V\hookrightarrow V_{(i_1,\dots,i_l)}\hookrightarrow \mathscr{R}V$ are finite birational extensions. From Observation \ref{correctdegree}(iii), $\mathscr{R}V$ is an $\mathrm{R}_1$-ification for $V$.
\par We now identify a finite birational overring of $\mathscr{R}V$ that is $S$-free. This ring would then inherit $R_1$ from $\mathscr{R}V$ by Observation \ref{correctdegree}(iii) and the proof would be complete, since a free $S$-module also satisfies Serre's condition $S_2$. The rest of the proof concerns identifying this overring.
\par Note that $A=S[\kappa_1,\dots,\kappa_r]$ is $S$-free of rank $p^r$ with a basis given by \[F:=\{\prod_{i=1}^n(\kappa_i-h_i)^{j_i}\:|\: 0\leq j_i\leq p-1\}\]
For each $1\leq i\leq r$, define $\Gamma_i:F\rightarrow \mathbb{N}\cup\{0\}$ by $\Gamma_i((\kappa_1-h_1)^{j_1}\dots(\kappa_i-h_i)^{j_i}\dots (\kappa_r-h_r)^{j_r})=j_i$ and \[\Gamma:F\rightarrow (\mathbb{N}\cup\{0\})^r, \; f\mapsto (\Gamma_1(f),\dots,\Gamma_r(f))\]
Let $\Omega: (\mathbb{N}\cup\{0\})^r\rightarrow \mathbb{N}\cup \{0\}$ be the map sending $(x_1,\dots,x_r)\mapsto x_1+\dots+x_r$. For every $0\leq k\leq r$, set \[\mathscr{V}_k:= \{p^{-k}\cdot m\:|\: m\in (\Omega\Gamma)^{-1}([(p-1)k,(p-1)(k+1)))\}\]
and $\mathscr{V}:=\cup_{0\leq k\leq r}\mathscr{V}_k$. By definition, the sets $\mathscr{V}$ and $F$ are in bijection and it follows that $\langle\mathscr{V}\rangle_S$ is $S$-free of rank $p^r$. Note that $A\subseteq \langle\mathscr{V}\rangle_S$ and the ring generators of $\mathscr{R}V$ over $A$ are all in $\mathscr{V}$. Suppose that $\langle\mathscr{V}\rangle_S$ is a $S$-algebra. Then $\mathscr{R}V\hookrightarrow \langle\mathscr{V}\rangle_S$ would be a birational module finite map of rings so that by Observation \ref{correctdegree}(iii), $R=\langle\mathscr{V}\rangle_S$.
 \par Thus, it only remains to be shown that $\langle\mathscr{V}\rangle_S$ is a $S$-algebra. We first note that it is an $A$-module. To see this, it suffices to show that multiplication by ring generators of $A$ over $S$ define an endomorphism on $\langle\mathscr{V}\rangle_S$. Set $E_i:=\langle\cup_{0\leq k\leq i}\mathscr{V}_k\rangle_A$ for each $0\leq i\leq r$. We proceed by induction on $i$. Clearly, $E_0\subseteq E$. Now assume that $E_{i-1}\subseteq E$ for some $1\leq i\leq r$. Consider $p^{-i}\cdot m\in \mathscr{V}_i$ and $\kappa_j-h_j$ for some $1\leq j\leq r$. If $\Gamma_j(m)\neq p-1$, then clearly $(\kappa_j-h_j)\cdot p^{-i}\cdot m\in \langle\mathscr{V}\rangle_S$. Suppose that $\Gamma_j(m)=p-1$. Let $c_j'$ denote the image in $S[\omega_j]$ of the element $C'$ in \Cref{P6} under the map sending $W\rightarrow \omega_j$, where we take $h=h_j$. From the relation (see Lemma \ref{P6})
 \begin{align}\label{class1EQ1}
 (\kappa_j-h_j)^p  &= \kappa_j^p-h_j^p+c'_jp(\kappa_j-h_j)\nonumber
 \\
 &\equiv c'_jp(\kappa_j-h_j)\:mod\:(p^2A)
 \end{align}
 we see that $(\kappa_j-h_j)\cdot p^{-i}\cdot m\in E_{i-1}\subseteq \langle\mathscr{V}\rangle_S$. Thus $E_i\subseteq \langle\mathscr{V}\rangle_S$ and by induction, $\langle\mathscr{V}\rangle_S$ is an $A$-module. To finish the proof, it suffices to show that the product of two elements from $\mathscr{V}$ lies in $\langle\mathscr{V}\rangle_S$. This is verified below. Consider $x:=p^{-u}m_u\in \mathscr{V}_u$ and $y:=p^{-v}m_v\in \mathscr{V}_v$. If $m_u\cdot m_v\in F$ then $xy\in \mathscr{V}$ and we are done. Suppose $m_u\cdot m_v\notin F$. Write
\[m_um_v=\alpha\cdot \prod_{j=l+1}^r(\kappa_{i_j}-h_{i_j})^{a_j}\]
where $0\leq l\leq r-1$ is such that $a_j\geq p$ for all $l+1\leq j\leq r$ and $\alpha\in F$. For each $l+1\leq j\leq r$, write using \Cref{class1EQ1}
\[(\kappa_{i_j}-h_{i_j})^p=p^2b_{i_j}-pc'_{i_j}(\kappa_{i_j}-h_{i_j})\]
for some $b_{i_j}\in S$.
 Setting $\beta:=\alpha\prod_{j=l+1}^r(\kappa_{i_j}-h_{i_j})^{a_j-p}$, we have that $\beta\in F$ and $\Gamma_j(\beta)\leq p-2$ for all $l+1\leq j\leq r$. We thus have
 \[xy=p^{-(u+v-r+l)}\beta\prod_{j=l+1}^r(pb_{i_j}-c'_{i_j}(\kappa_{i_j}-h_{i_j}))\]
We verify that every monomial in the above expression lies in $\langle\mathscr{V}\rangle_S$. First consider
\[
p^{-(u+v-r+l)}\beta\prod_{j=l+1}^r(pb_{i_j})\in S\cdot p^{-(u+v-2r+2l)}\beta\]
We have 
\[(p-1)u+(p-1)v-(2p-2)(r-l)\leq (p-1)u+(p-1)v-p(r-l)\leq \Omega\Gamma(\beta)\]
so this implies $S\cdot p^{-(u+v-2r+2l)}\beta\subseteq \langle\mathscr{V}\rangle_S$. Next, consider the monomial \[p^{-(u+v-r+l)}\beta\prod_{j=l+1}^r(-1)c'_{i_j}(\kappa_{i_j}-h_{i_j}).\] Now $\beta':=\beta\prod_{j=l+1}^r(\kappa_{i_j}-h_{i_j})\in F$ since $\Gamma_j(\beta)\leq p-2$ for all $l+1\leq j\leq r$. Since $\langle\mathscr{V}\rangle_S$ is an $A$-module, it suffices to show $p^{-(u+v-r+l)}\beta'\in \langle\mathscr{V}\rangle_S$. But this is the case since
\[(p-1)(u+v-r+l)\leq \Omega\Gamma(m_u)+\Omega\Gamma(m_v)-(p-1)(r-l)=\Omega\Gamma(\beta')\]
Finally, for the case of a general monomial, it suffices to show by symmetry that for each $l+1\leq t\leq r-1$ 
\[p^{-(u+v-r+l)}\beta(\prod_{j=l+1}^tpb_{i_j})(\prod_{j=t+1}^r(-1)c'_{i_j}(\kappa_{i_j}-h_{i_j}))\in \langle\mathscr{V}\rangle_S\]
Since $\mathscr{V}\rangle_S$ is an $A$-module, it suffices to show that $p^{-(u+v+2l-r-t)}\beta\prod_{j=t+1}^r(\kappa_{i_j}-h_{i_j})\in \langle\mathscr{V}\rangle_S$. Note that $\beta'=\beta\prod_{j=t+1}^r(\kappa_{i_j}-h_{i_j})\in F$. We have
\begin{align*}
    \Omega\Gamma(\beta')&=\Omega\Gamma(m_u)+\Omega\Gamma(m_v)-p(r-l)+(r-t) \\
    &\geq (p-1)u+(p-1)v-p(r-l)+(r-t)\\
    &=(p-1)(u+v-r)+pl-t \\
    &\geq (p-1)(u+v-r)+pl-t+(p-2)(l-t) \\
    &=(p-1)(u+v-r)+(2p-2)l-(p-1)t
\end{align*}
so that $p^{-(u+v+2l-r-t)}\beta'\in \langle\mathscr{V}\rangle_S$ and thus the proof is complete.

\medskip
Now suppose that $p=2$. From the proof of [\cite{DK1}, Lemma 3.2] we know that the integral closure of $S[\kappa_i]$ is $R_i=S[\tau_i]$ for $\tau_i:= \frac{1}{2}\cdot(\kappa_i+h_i)$. It also tells us that if $f_i=h_i^p+4g_i$, then $\tau_i$ satisfies $l_i(T):=T^2-h_iT-g_i\in S[T]$, where $T$ is an indeterminate over $S$. Since $l_i(T)$ and $l_i'(T)$ are relatively prime over the quotient field of $S/2S$, $2\in S[\tau_i]$ is square free. Finally, we also know that $R_i$ is $S$-free for all $i$. Let $V$ denote the join of the $R_i$, that is $V=S[\tau_1,\dots,\tau_r]$. By \Cref{kernel_lem}, $V$ is a free $S$-module. Note that $V$ is birational to $R$ and satisfies Serre's criterion $(S_2)$. From Observation \ref{correctdegree}(i), we have that $V[1/2]$ is integrally closed. Moreover, $2\in V$ is square-free since for each $2\leq i\leq r$, $l_i(T)$ and $l_i'(T)$ are relatively prime over the quotient field of $S[\tau_1,\dots,\tau_{i-1}]/Q$ for all height one primes $Q\subseteq S[\tau_1,\dots,\tau_{i-1}]$ containing $2$. Thus $V$ is regular in codimension one and hence $V=R$ is a free $S$-module.
 \end{proof}

 \medskip
 Here is the main theorem of this paper.
 
 \begin{thm}\label{class1towers} Let $S$ be an integrally closed domain with fraction field $L$ and $p$ a prime integer that remains a non-zero prime in $S$. Let
$f_1, \ldots, f_r$ be square-free elements such that $K := L(\sqrt[n_1]{f_1}, \ldots, \sqrt[n_r]{f_r})$ is a class one radical tower, where, for each $1\leq i\leq r$, $n_i=pd_i$ and $d_i\in S$ is a unit. Then the integral closure of $S$ in $K$ is a free $S$-module. If $S$ is Cohen-Macaulay or an unramified regular local ring, then the integral closure of $S$ in $K$ is Cohen-Macaulay. 
          \end{thm}
\begin{proof} Write $n_i = pd_i$, so that $d_i$ is a unit in $S$, for all $1\leq i\leq r$. Note that this condition is the same as assuming $p\mid n_i$ but $p^2\nmid n_i$, if $S$ were local. Set $\omega_i: = (\sqrt[n_i]{f_i})^{d_i}$, so that $\omega _i$ is a $p$th root of the square-free element $f_i$. Let $R_0$ denote the integral closure of $S$ in $K_0 := L(\omega_1, \ldots, \omega _r)$, a class one radical extension of $L$. Then, by Theorem \ref{class1thm}, $R_0$ is free over $S$. We now claim that each $\omega _i$ is square-free in $R_0$. 
Set $A_0 := S[\omega_1, \ldots, \omega_r]$. Since $p$ does not divide any $f_i$, if $Q\subseteq R_0$ is a height one prime containing $\omega_i$, then $p\not \in Q$. Thus, $(R_0)_Q = (R_0[\frac{1}{p}])_Q = (A_0)_{Q\cap A_0}$. Thus, it suffices to show that each $\omega _i$ is square-free in $A_0$. Without loss of generality, we may assume $i = 1$. 

\medskip
\noindent
Let $Q$ be a height one prime in $A_0$ containing $\omega _1$. Since $p\not \in Q$, it suffices to show that $\omega _1$ is square-free in $B[\omega_1]$, where $B := S[\frac{1}{p}, \omega_2, \ldots, \omega_r]$. By Proposition \ref{HK}, $B$ is integrally closed and $f_1$ is square-free in $B$. By Observation \ref{correctdegree}(ii), $B[\omega_1] \cong B[X]/\langle X^p-f_1\rangle$. Thus, $Q$ corresponds to a height two prime $Q'$ in $B[X]$ containing $X$ and $X^p-f_1$, and thus, $Q'$ is a height two prime containing $X$ and $f_1$. It follows that $Q' = \langle X, P\rangle$, where $P\subseteq B$ is a height one prime containing $f_1$. Since $PB_P = f_1B_P$, we have $Q'B[X]_{Q'} = (X, f_1)B[X]_{Q'}$. Therefore, modding out $X^p-f_1$ gives $QB[\omega_1]_Q = \omega_1B[\omega_1]_Q$, which is what we want. 

Now, we may regard each $\sqrt[n_i]{f_i}$ as a $d_i$th root of $\omega _i$, and we write $\sqrt[d_i]{\omega_i} = \sqrt[n_i]{f_i}$. If we let $K_0$ denote the quotient field of $R_0$ and $R$ denote the integral closure of $S$ in $K$, then $R$ is the integral closure of $R_0$ in $K_0(\sqrt[d_1]{\omega_1}, \ldots, \sqrt[d_r]{\omega_r})$. Each $d_i$ is a unit in $R_0$ and by the paragraph above, $\omega _i$ is square-free in $R_0$. Moreover, the going down property in the extension $S\subseteq R_0$ ensures that the set $\{\omega_1, \ldots, \omega_r\}$ satisfies $\msa$ in $R_0$. Thus, by Proposition \ref{HK}, $R = R_0[\sqrt[d_1]{\omega_1}, \ldots, \sqrt[d_r]{\omega_r}]$. It follows that $R$ is a free $R_0$-module. Since $R_0$ is a free $S$-module, we have that $R$ is a free $S$-module, which is what we want. The second statement follows immediately from the first.
\end{proof}

 \section{Applications} In this section we focus our attention on applications of the results from the previous section to the case that $S$ is an unramified regular local ring of mixed characteristic $p > 0$. Our first applications of Corollary \ref{class1towers} deal with cases where we adjoin $n$th roots of elements in $S$ such that the resulting extension need not be a class one radical extension. The resulting integral closure $R$ may not be Cohen-Macaulay, but does have a small Cohen-Macaulay algebra, i.e., a module finite $R$-algebra $\tilde{R}$ such that $\tilde{R}$ is Cohen-Macaulay. 
 
For the theorem below, we assume $S$ to be an unramified regular local ring of mixed characteristic $p$. Let $\{p, x_2, \ldots, x_d\} = \{p, \un{x}\}$ be a minimal generating set for the maximal ideal of $S$. For all $k\geq 1$, we set $T_k(\un{x})  := S[\sqrt[k]{x_2}, \ldots, \sqrt[k]{x_d}]$. We choose $k$th roots of the $x_i$ so that if $k = ab$, then $(\sqrt[k]{x_i})^a = \sqrt[b]{x_i}$, for all $i$. Thus, if $k\mid l$, then $T_k (\un{x})\subseteq T_l(\un{x})$. It follows from Observation \ref{correctdegree}(ii) that $T_k(\un{x})$ is an unramified regular local ring of mixed characteristic $p$, for all $k\geq 1$. Set $W(\un{x}) := \bigcup_{k\geq 1} (T_{k}(\un{x})^{p\wedge p^2}\cap S)$. It is not difficult to see that if we take the union of $W(\un{x})$, as $\{\un{x}\}$ ranges over regular system of parameters for $S$, this set is considerably larger than $\swp$. For example, if $f = m+h^p$, where $m$ is a monomial in $x_2, \ldots, x_d$ , which are part of a regular system of parameter for $S$, then $f \in W(\un{x})$. 

 \begin{thm}\label{app1} Let $S$ be an unramified regular local ring of mixed characteristic $p$. Assume $f_1, \ldots, f_r$ are square-free, satisfy $\msa$ and belong to $W(\un{x})$, for $\un{x} = x_2, \ldots, x_d$ such that $p, x_2, \ldots, x_d$ is a regular system of parameters. Let $\omega_i^{n_i} = f_i$ be roots in some field extension of $L$ such that $p\mid n_i$ and $p^2\nmid n_i$ for each $i$, and  write $R$ for the integral closure of $S$ in $L(\omega_1,\dots,\omega_r)$. Then $R$ admits a small  Cohen-Macaulay algebra.
\end{thm} 

\begin{proof} Set $T_k = T_k(\un{x})$, for all $k$. The strategy is to replace $S$ by $T_k$ for some $k$ and apply Theorem \ref{class1thm}. Because we wish to pass to $T_k$ for some $k$, $f_j$ will not be square-free in $T_k$ if $f_j$ is divisible by some $x_i$. So we first need to remove any factor of $x_i$ appearing in any $f_j$. After re-indexing, we may assume that $x_2, \ldots, x_t$ are exactly those $x_i$ appearing as factors among the $f_j$. Because the elements $f_1, \ldots, f_r$ satisfy $\msa$, for any $x_i$ with $2\leq i\leq t$, there is exactly one $f_j$ such that $x_i$ divides $f_j$. We may re-index the $f_j$ to assume $x_i$ divides $f_{i-1}$, for $2\leq i\leq t$, and therefore no $x_i$ divides any $f_j$ if $i > t$. For $2\leq i\leq t$, write $f_{i-1} = x_if'_{i-1}$. 

\medskip
\noindent
Now, by assumption, each $f_j\in T_{k_j}^{p\wedge p^2}$ for some $k_j$, with $1\leq j\leq r$. If we take $k := k_1\cdots k_r\cdot p$, then we have $f_j \in T_k^{p\wedge p^2}$, for all $j$, and moreover, $p\mid k$. We can do this since if $a\mid b$, then $T_a^{p\wedge p^2} \subseteq T_b^{p\wedge p^2}$. 

\medskip
\noindent
We start by removing the factor $x_2$ from $f_1$. We wish to prove that $f_1' \in T_k^{p\wedge p^2}$. Now, by assumption, in $T_k$ we have $f_1 = h^p+bp^2$. 
Thus, \[(x_2f_1') = (\sqrt[p]{x_2})^pf_1' = h^p+bp^2,\] so that in $T_k/pT_k$, we have $(\sqrt[p]{x_2})^pf_1' \equiv h^p$. Since $T_k/pT_k$ is ia UFD, it follows that in $T_k/pT_k$, $f_1'$ is a $p$th power, say $f_1'\equiv h_0^p$. Thus, in $T_k$, $f_1' = h_0^p+ap$, for some $a \in T_k$. Therefore, $f_1 = x_2h_0^p+x_2ap = h^p+bp^2$. It follows that, in $T_k/pT_k$, we have 
\[
0\equiv h^p-x_2h_0^p \equiv (h-\sqrt[p]{x_2}\ h_0)^p,\]
and thus, $h \equiv \sqrt[p]{x_2}\ h_0$ mod $pT_k$, so that $h = \sqrt[p]{x_2}\ h_0+cp$, for some $c\in T_k$. Thus, \[(\sqrt[p]{x_2}\ h_0+cp)^p+bp^2 = x_2h_0^p +ax_2p.\]
Subtracting $x_2h_0^p$ from both sides of this last equation, we have that $p^2$ divides the left hand side of the equation, so $p^2$ divides the right hand side of the equation, giving $p\mid a$. Thus, $f_1' \in T_k^{p\wedge p^2}$, as required.  We now note that the set $C := \{\sqrt[k]{x_2}, f_1', f_2, \ldots, f_r\}$ satisfies $\msa$ in $T_k$. This follows since a height one prime in $T_k$ contracts to a height one prime in $S$, and the set $\{x_2 , f_1', f_2, \ldots, \ldots, f_r\}$ satisfies $\msa$ in $S$. Note however, that for $j\geq 2$, the $f_j$ need not be square-free in $T_k$, since $f_2, \ldots, f_{t-1}$ are divisible by $x_2, \ldots, x_t$, respectively. 

\medskip
\noindent
Now, we repeat the process removing $x_3$ from the element $f_2$. The same argument as above, shows $f_2' \in T_k^{p\wedge p^2}\cap S$ and $\{\sqrt[k]{x_2}, \sqrt[k]{x_3}, f_1', f_2', f_3, \ldots, f_r\}$ satisfies $\msa$ in $T_k$. Continuing this process we eventually arrive at the following set up: We have 
\begin{enumerate}
\item[(a)] $\go, \ldots, \gr \in T_k^{p\wedge p^2}\cap S$ and no $\gj$ is divisible in $S$ (or $T_k$) by any $x_i$. 
\item[(b)] $\gj = f_j$ for $t\leq j\leq r$ and $f_j = x_j\gj$, where $\gj = f_j'$, for $1\leq j\leq t-1$. 
\item[(c)] The set $\{\sqrt[k]{x_2}, \ldots, \sqrt[k]{x_t}, \go, \ldots, \gr\}$ satisfies $\msa$.
\end{enumerate} 

\medskip
\noindent
We now observe that $\go, \ldots, \gr$ are square-free elements in $T_k$, and not divisible by $p$. The latter statement holds, since the $\gj$ are not divisible by $p$ in $S$, and therefore are not divisible by $p$ in $T_k$, since $pS = pT_k\cap S$. To see that $\go, \ldots, \gt$ are square-free in $T_k$, it suffices to show they are square-free in $T_k[\frac{1}{p}]$, since $p$ does not divide any $\gj$. However, the set $\{\sqrt[k]{x_2}, \ldots, \sqrt[k]{x_d}, \go, \ldots, \gr\}$ satisfies $\msa$, so \Cref{HK} gives us what we want.

\medskip
\noindent
 For ease of notation, set $\sqrt[k]{x_i} := y_i$, for $2\leq i\leq d$. We make one more ring extension. Set $\tilde{T} := T_k[\sqrt[n_1]{y_1}, \ldots, \sqrt[n_t]{y_t}]$ and let $\tilde{E}$ denote its quotient field. Here we are choosing $\sqrt[n_i]{y_i}$ so that $(\sqrt[n_i]{y_i})^k = \sqrt[n_j]{x_i}$. Then $\tilde{T}$ is an unramified regular local ring of mixed characteristic $p$, and since $\go, \ldots, \gr$ are square-free in $T_k$, satisfy $\msa$, and each $\gj \in T_k^{p\wedge p^2}$, they remain such in $\tilde{T}$. It follows that $F :=  \tilde{E}(\sqrt[n_1]{\go}, \ldots, \sqrt[n_r]{\gr})$ is a class one radical extension of $\tilde{E}$. Thus, by our assumption on the $n_j$, if we let $\tilde{R}$ denote the integral closure of $\tilde{T}$ in 
$F$, then Corollary \ref{class1towers} implies that $\tilde R$ is Cohen-Macaulay. The proof of the theorem will be complete, once we observe that $R\subseteq \tilde{R}$, and for this it suffices to see that $L(\omega_1, \ldots, \omega_r)\subseteq F$. 

\medskip
\noindent
To finish, we note that, since $f_j = x_{j+1}\gj$, for $1\leq j\leq t-1$ and $f_j = \gj$, for $t\leq j\leq r$,  \[L(\omega_1, \ldots, \omega_r)\subseteq  
L(\sqrt[n_1]{x_2}, \ldots, \sqrt[n_t]{x_t}, \sqrt[n_1]{\go}, \ldots, \sqrt[n_r]{\gr}).\] Since $(\sqrt[n_i]{y_i})^k = \sqrt[n_j]{x_i}$, $L(\sqrt[n_1]{x_2}, \ldots, \sqrt[n_t]{x_t}, \sqrt[n_1]{\go}, \ldots, \sqrt[n_r]{\gr}) \subseteq F$, which gives what we want. 
\end{proof}

\begin{rem} In Theorem \ref{app1}, we can remove the restrictions that $f_1, \ldots, f_r$ be square-free and satisfy $\msa$, by allowing extra factors involving $x_2, \ldots,x_d$. For example, if $f_1, \ldots, f_r \in W(\un{x})$, and each $f_j = m_j\gj$, where $m_j$ is a monomial in $x_2 \ldots, x_d$ and $\go, \ldots, \gr$ are square-free and satisfy $\msa$, the same proof will show that $\go, \ldots, \gr$ belong to $W(\un{x})$ and that there is a finite extension $\tilde{T}$ of $R$ such that $\tilde{T}$ is an unramified regular local ring of mixed characteristic $p$ with quotient field $\tilde{E}$, so that $F := \tilde{E}(\sqrt[n_1]{\go} \ldots, \sqrt[n_r]{\gr})$ is a class one radical extension of $\tilde{E}$.\footnote{In this scenario, $\go, \ldots, \gr$ will not be square-free and satisfy $\msa$ until all $x_i$ have been removed from all $f_j$.} Thus, the integral closure $\tilde{R}$ of $\tilde{T} $ in $\tilde{E}$ is Cohen-Macaulay, and therefore a small Cohen-Macaulay algebra for $R$.  
\end{rem}

\begin{cor}\label{app2} Let $S$ be an unramified regular local ring of mixed characteristic $p$, and $g_1, \ldots, g_s \in S$, not divisible by $p$ (and not necessarily square-free) and suppose that in $S$ we can write each $g_i := q_{i_1}^{c_{i_1}}\cdots q_{i_t}^{c_{i_t}}$ as a product of primes $q_{i_j}$. Let $\omega_i^{n_i} = g_i$ be roots in some field extension of $L$ such that $p\mid n_i$ and $p^2\nmid n_i$ for each $i$. Let $R$ the integral closure of $S$ in $K := L(\omega_1,\dots,\omega_r)$. Suppose there exists $W(\un{x})$ as in Theorem \ref{app1} such that each $q_{i_j} \in W(\un{x})$. Then $R$ admits a small Cohen-Macaulay algebra. 
\end{cor}

\begin{proof} We first point out that our elements $g_j$ are products of primes and not a unit times a product of primes. We begin with the reduction used in the proof of Theorem 6.1 in \cite{HK}, which we repeat for the reader's convenience. Write each $g_i := q_{i_1}^{d_{i_1}n_i+e_i} \cdots q_{i_t}^{d_{i_t}n_i+e_{i_t}}$. Then for \[\gamma_i := q_{i_1}^{d_{i_1}}(\sqrt[n_i]{q_{i_1}})^{e_{i_1}}\cdots q_{i_t}^{d_{i_t}}(\sqrt[n_i]{q_{i_t}})^{e_{i_t}},\] $\gamma_i ^{n_i} = g_i$, so $S[\sqrt[n_i]{g_i}] \subseteq S[\sqrt[n_i]{q_{i_1}},\ldots, \sqrt[n_i]{q_{i_t}}]$, and hence $L(\sqrt[n_i]{g_i})\subseteq L(\sqrt[n_i]{q_{i_1}},\ldots, \sqrt[n_i]{q_{i_t}})$. Doing this for each $g_i$ shows $K$ is contained in $F := L(\sqrt[n_1]{q_1},\ldots, \sqrt[n_r]{q_r})$. Thus, $R$ is contained in the integral closure, say $T$, of $S$ in $F$. Since $q_1, \ldots, q_r$ are distinct primes in $S$, they are square-free and satisfy $\msa$, and by assumption, they belong to $V$. Thus, by \Cref{app1}, $T$ admits a small Cohen-Macaulay algebra, which in turn, is also a small Cohen-Macaulay algebra for $R$, which completes the proof. 
\end{proof}

\begin{cor} Let $S$ be an unramified regular local ring of mixed characteristic $p$ and $g_1, \ldots, g_r \in S$, square-free and satisfying $\msa$. Let $n_1, \ldots, n_r$ be as in Theorem \ref{app1}. Assume each $g_i = m_i+b_ip^2$, where $m_i$ is a monomial in $x_2, \ldots,x_d$ and $b_i \not = 0$. Then $R$ admits a small Cohen-Macaulay algebra.
\end{cor} 

\begin{proof} If $h = x_2^{e_2}\cdots x_d^{e_d}$ is a monomial in $x_2, \ldots, x_d$, then $m = \{(\sqrt[p]{x_1})^{e_1}\cdots (\sqrt[p]{x_d})^{e_d}\}^p$ in $T_p$. 
Thus, each $g_j \in W(\un{x})$, and the result follows from Theorem \ref{app1}. 
\end{proof}.

\begin{ex}\label{examples} The above results provide small Cohen-Macaulay algebras for many non Cohen-Macaulay rings. Here is a concrete example. Consider Example 3.10 in \cite{DK1}. We start with an unramified regular local ring of mixed characteristic $3$, say $S$, and take $x,y\in S$ such that $3,x,y$ form part of a regular system of parameters for $S$. Take $a:=xy^4+9$, $b:=x^4y+9$, $f=ab^2$ and $\omega$ a cube root of $f$. If $R$ is the integral closure of $S$ in the quotient field of $S[\omega]$, then, as shown in \cite{DK1}, $R$ is not Cohen-Macaulay. Set $T:=S[\sqrt[3]{x},\sqrt[3]{y}]$. Then $T$ is an unramified regular local ring such that $3,\sqrt[3]{x},\sqrt[3]{y}$ form a part of a minimal generating set of the maximal ideal of $S$. Then $a,b\in T$ are square free, mutually co-prime and lie in $T^{3\wedge 9}$. By \Cref{class1towers}, the integral closure of $T$ in the quotient field of $T[\sqrt[3]{a},\sqrt[3]{b}]$, say $\tilde{R}$, is Cohen-Macaulay. Since $\tilde{R}$ is a module finite extension of $R$, it is a small Cohen-Macaulay algebra for $R$. 
    \par We note two things about the constructions in this paper: 
    
\noindent (i)  In the above example, $f\in S^{3\wedge 9}$, but the integral closure of $S$ in the quotient field of $S[\omega]$ is not Cohen-Macaulay. So, the square-free hypothesis in \Cref{class1thm} cannot be dropped.
 
 \medskip
 \noindent
 (ii) The integral closure in a general square free $p$-th root tower with elements from $S^{p\wedge p}$ need not be Cohen-Macaulay. Consider Example 4.10 in \cite{PS2}: Setting $S:=\mathbb{Z}[X,Y]_{(p,X,Y)}$ for some prime number $p\geq 3$, take $f_1:=X^{2p}-pX^{2p}+p^2$ and $f_2:=(XY)^p+p(XY)^p+p^2$. Then $f_1,f_2\in S^{p\wedge p}\setminus S^{p\wedge p^2}$ are square free and mutually coprime. If $R$ is the integral closure of $S$ in $L(\sqrt[p]{f_1},\sqrt[p]{f_2})$, then $\projdim_S(R)=1$.
\end{ex}

\begin{ex}
Here is an example illustrating \Cref{class1thm}. Take $S=\mathbb{Z}_3[[X,Y]]$, where $\mathbb{Z}_3$ denotes the $3$-adic integers and $f:=X^3+9$, $g=Y^3+9$. Then $f,g\in S^{3\wedge 9}$ are square free and mutually coprime. Let $L$ denote the fraction field of $S$ and let $R$ be the integral closure of $S$ in $L(\sqrt[3]{f},\sqrt[3]{g})$. Then $R$ is a free $S$-module of rank $9$ with a basis given by $\mathscr{V} = \mathscr{V}_0\cup \mathscr{V}_1\cup \mathscr{V}_2$, where
\[\mathscr{V}_0:=\{1, \sqrt[3]{f}-X, \sqrt[3]{g}-Y\}\]
\[\mathscr{V}_1:=\{3^{-1}(\sqrt[3]{f}-X)^2,3^{-1}(\sqrt[3]{f}-X)(\sqrt[3]{g}-Y),3^{-1}(\sqrt[3]{g}-Y)^2,3^{-1}(\sqrt[3]{f}-X)(\sqrt[3]{g}-Y)^2,3^{-1}(\sqrt[3]{f}-X)^2(\sqrt[3]{g}-Y)\}\]
\[\mathscr{V}_2 =\{9^{-1}(\sqrt[3]{f}-X)^2(\sqrt[3]{g}-Y)^2\}\]
\end{ex}

\section{Concluding remarks} 
\medskip
In this section we make a few remarks concerning our assumptions, and how they might be successfully altered. As mentioned in the introduction, our primary assumption that the elements whose roots we take belong to $\swp$ stems from the success we have had with this assumption in \cite{DK2}, \cite{PS1}, and \cite{PS2}, together with the fact that there is an unramified prime over $p$ with trivial extension of residue fields. In \cite{DK1}, for square-free $f$ not a $p$th power modulo $pS$, i.e., $f\not \in S^{p\wedge p}$, it is noted that $S[\omega]$ is already integrally closed. Unfortunately, as shown in \cite{PS1} and \cite{PS2}, adjoining $p$th roots of even two square-free elements that are not in $S^{p\wedge p}$ presents considerable difficulties. In particular, if $\omega ^p = f$ and $\mu ^p = g$, then $S[\omega, \mu]$ need not be integrally closed, nor need its integral closure be Cohen-Macaulay, when $S$ is regular. However, we can extend Theorem \ref{class1thm} by adjoining multiple elements that are not in $S^{p\wedge p}$ if we assume they are sufficiently independent  modulo $pS$. Suppose that we start with a class one tower over $S$ as in the statement of \Cref{class1thm}. and suppose that $S\hookrightarrow T$ is an integral extension of integrally closed domains such that $T$ is $S$-free, $p\in T$ is a principal prime and the $f_i\in T$ remain square free. Write $E$ for the fraction field of $T$. Then applying \Cref{class1thm} to $T$, the integral closure of $T$ in $E(\sqrt[p]{f_1},\dots,\sqrt[p]{f_r})$ is $T$-free. Therefore, the integral closure of $S$ in $E(\sqrt[p]{f_1},\dots,\sqrt[p]{f_r})$ is $S$-free. This is exactly what happens when we adjoin $p$-th roots of square-free elements \textit{linearly disjoint modulo $p$} alongside a class one tower.

\begin{dfn} Let $F$ denote the fraction field of $S/pS$. 
 The elements $g_1,\dots,g_t$ in $S$ are said to be \textit{linearly disjoint modulo $p$} if $g_i\notin F^p$ for all $i$ and there is an isomorphism of $F$-vector spaces \[F(\sqrt[p]{\bar{g_1}})\otimes_F\dots\otimes_FF(\sqrt[p]{\bar{g_t}})\simeq F(\sqrt[p]{\bar{g_1}},\dots,\sqrt[p]{\bar{g_t}}),\] where $\bar{g_i}$ is the image of $g_i$ modulo $p$.
\end{dfn}

Note that if $S/p$ is integrally closed, then $\bar{g} \not \in F^p$ if and only if $g_i\notin S^p$. The following proposition extends the crucial result Theorem \ref{class1thm} from section three. 

\begin{prop}\label{disjmodp}
Let $S$ be an integrally closed domain with fraction field $L$ such that $p\in S$ is a non-zero principal prime. Let $f_1,\dots,f_r,g_1,\dots,g_t\in S$ be square-free elements satisfying $\mathscr{A}_1$ such that that $L(\sqrt[p]{f_1},\dots,\sqrt[p]{f_r})$ is a class one tower and the $g_i$ are linearly disjoint modulo $p$. Then, the integral closure of $S$ in $K := L(\sqrt[p]{f_1},\dots,\sqrt[p]{f_r},\sqrt[p]{g_1},\dots,\sqrt[p]{g_t})$ is a free $S$-module. In particular, if $S$ is Cohen-Macaulay or an unramified regular local ring, then the integral closure of $S$ in $K$ is Cohen-Macaulay.
\end{prop}

\medskip
\noindent

\begin{proof}
By \Cref{correctdegree}(ii), $T:=S[\sqrt[p]{g_1},\dots,\sqrt[p]{g_t}]\simeq S[X_1,\dots,X_t]/(X_1^p-g_1,\dots,X_t^p-g_t)$ where the $X_i$ are indeterminates over $S$. The linear disjointness modulo $p$ hypothesis on the $g_i$ implies that $p\in T$ is a principal prime. By \Cref{correctdegree}(i), $T[1/p]$ is integrally closed. Since $T$ is $S$-free, these facts imply that $T$ is integrally closed. To conclude that the $f_i\in T$ are square free, we can assume $p\in S$ is a unit, since $f_i\notin pS$ for all $i$. Applying \Cref{HK} to $S[1/p]$, we see that the $f_i\in T$ are all square-free. Thus the $f_i$ define a class one tower over $T$ and hence by \Cref{class1thm}, the integral closure of $T$ in $L(\sqrt[p]{f_1},\dots,\sqrt[p]{f_r},\sqrt[p]{g_1},\dots,\sqrt[p]{g_t})$ is a free $T$-module and hence a free $S$-module.
\end{proof}

\begin{rem} With Proposition \ref{disjmodp} in hand, it is not difficult to see that Theorem \ref{class1towers}  can be extended in a similar way. \end{rem} 

When $S$ is a complete unramified regular local ring of mixed characteristic $p >0$, whose residue field is $F$-finite, there is an alternate way of dealing with elements that do not belong to $S^{p\wedge p}$, as the following proposition shows.

\begin{prop}\noindent\label{pwedge_p2}
Let $(S,\m,k)$ be a complete unramified regular local ring of mixed characteristic $p>0$ with $k$ $F$-finite. Then there exists an unramified regular local ring  $T$ of mixed characteristic $p>0$ such that $T$ is finite over $S$ and $S\hookrightarrow T^{p\wedge p} $.
\end{prop}
\begin{proof}\noindent
Complete $p$ to a minimal system of generators of $\m$, say $\m=(p,x_2,\dots,x_d) = (p, \un{x})$ and let $x_i'$ denote the image of $x_i$ in $S/pS$. Let $F$ denote the Frobenius map on $S/pS$. By hypothesis, $E:=S/pS$ is an $F$-finite regular local ring and $E^{1/p}$ is obtained by adjoining to $E$, the $p$-th roots of the $x_i'$ and the $p$-th roots of a basis of $k$ over $k^p$. Take $T$ to be the $S$-algebra obtained by adjoining the $p$-th roots of the $x_i$ and the $p$-th roots of a fixed set of lifts of a basis of $k$ over $k^p$. By construction, $S\hookrightarrow T^{p \wedge p}$. Moreover, it follows easily that $T$ is regular local with maximal ideal generated by $(p,\sqrt[p]{x_2},\dots,\sqrt[p]{x_d})$ and residue field $k^{1/p}$.  \footnote{That $T$ is regular can be checked one element at a time. After adjoining the $p$-th roots of the $x_i$, we get a URLR. When we adjoin the $p$-th roots of the units in question, the generators of the maximal ideal remain the same at each step.}
\end{proof}

\begin{rem} Let $S$, $\un{x}$ and $T$ be as in the preceding proposition, so that $S\subseteq T^{p\wedge p}$. Assume further that the residue field $k$ is perfect. Thus, $T = T_p(\un{x})$, where for $k\geq 1$, $T_k(\un{x})$ has the same meaning as in Theorem \ref{app1}. We want to observe, that in this case, $T_p(\un{x})\cap S = \bigcup_{k\geq 1}(T(\un{x})^{p\wedge p^2}\cap S)=W(\un{x})$. Suppose $f = h^p+ap^2$, with $h, a \in T_k(\un{x})$, for some $k>p$, and $f\in S$ square-free. 
We can write $f = h_0^p+bp$, for $h_0, b\in T$. Then $0 \equiv h^p-h_0^p \equiv (h-h_0)^2$ mod $pT_k(\un{x})$, from which it follows that $h \equiv h_0$ mod $pT_k(\un{x})$. 
Thus, $h = h_0 +cp$, for some $c\in T_k(\un{x})$. Therefore $f = (h_0+cp)^p + ap^2 = h_0^p+dp^2$, for some $d \in T_k(\un{x})$. Thus $f-h_0^p \in p^2T_k(\un{x})\cap T = p^2T$, since $T$ is integrally closed and $T_k(\un{x})$ is integral over $T$. It follows that $f\in T^{p\wedge p^2}$, which is what we want. 
\end{rem}

Consider the question of the existence of maximal Cohen-Macaulay modules over the integral closure of a complete regular local ring $S$ of mixed characteristic $p>0$ with perfect residue field, in an arbitrary $p$-th root tower over its quotient field $K$. One of the main differences between our point of view in this paper and our earlier ones, is that we do not restrict ourselves to {\it birational} Cohen-Macaulay modules or algebras (see Theorem \ref{app1}), wheres the constructions in \cite{DK1}, \cite{PS1}, \cite{PS2} are all birational. Thus, results like Proposition \ref{pwedge_p2} simultaneously allow us to move beyond the birational assumption and also enable us to restrict attention to the case of general square free towers with elements chosen from $S^p$. Indeed, choose an unramified regular local ring $T$, such that $S\hookrightarrow T^{p\wedge p}$ as in \Cref{pwedge_p2}. Let $L'$ be the 
fraction field of $T$ and set $\mathscr{K}:=K[L']$. Note that \Cref{HK} implies that square free elements of $S$ that are not divisible by the regular system of parameters used to construct $T$ from $S$ remain square free in $T$. It is then clear that $\mathscr{K}$ embeds in a finite field extension of it, say $\tilde{\mathscr{K}}$, obtained by adjoining $p$-th roots of mutually coprime square free elements of $T$ coming from $T^{p\wedge p}$ to $L'$. Let $\mathscr{R}$ be the integral closure of $S$ in $\tilde{\mathscr{K}}$. It then suffices to show that $\mathscr{R}$ admits a maximal Cohen-Macaulay module. In summary, in attempting to construct a maximal Cohen-Macaulay module or algebra over $R$,  we can simply start over and assume that the elements whose roots we adjoin are mutually coprime, square free and come from $S^{p\wedge p}$.

\medskip 
Finally we address the question of units. As pointed out in the introduction, when the exponents $n_i$ are units, the ring $S[\sqrt[n_1]{f_1}, \ldots, \sqrt[n_r]{f_r}]$ is integrally closed and a free $S$-module. It is not difficult to show that if, say, in Theorem \ref{class1towers} we allow some of the $n_i$ to be units, while the other $n_i = pd_i$, with $d_i$ a unit, then the conclusion of the theorem still holds. The conclusions of our other results then follow in this case as well. However, we have focused on the case that the $n_i$ are non-units, since we are motivated by the case where $S$ is an unramified regular local ring of mixed characteristic $p$, and the extensions of quotient fields have degree divisible by $p$.  We could also consider the case that some of the $f_i$ are units. However, this potentially introduces roots of unity and we lose the property that the degree of the extension of quotients field equals $p^r$. We address the issue involving roots of unity in a forth coming paper, \cite{KS}.

\end{document}